\documentclass[a4paper,10pt,leqno]{amsart}
\usepackage{amsmath,amsfonts,amssymb,amsthm,epsfig,epstopdf,url,array}
\usepackage{amsthm}
\usepackage{mathrsfs}
\usepackage{epsfig}
\usepackage{graphicx}
\usepackage{tikz}
\usetikzlibrary{trees}

\newtheorem{thm}{Theorem}[section]
\newtheorem{prop}[thm]{Proposition}

\newtheorem{lem}[thm]{Lemma}
\newtheorem{rmk}[thm]{Remark}

\numberwithin{equation}{section}

\newcommand{\ord}{\text{ord}}
\newcommand{\N}{{\mathbb{N}}}
\newcommand{\z}{{\mathbb{Z}}}
\newcommand{\q}{{\mathbb{Q}}}

\newcommand{\eq}{\equiv}
\newcommand{\rom}[1]{\uppercase\expandafter{\romannumeral #1\relax}}

\title[Representation of $m$-goanl forms over $\N_0$]{Representation of $m$-gonal forms over $\N_0$ and a finiteness theorem for universal original version's $m$-gonal forms}

\author{Dayoon Park}
\address{Department of Mathematical Sciences, Ulsan National Institute of Science and
Technology, Ulsan, Korea}
\email{pdy1016@unist.ac.kr}
\thanks {This work was supported by the National Research Foundation of Korea (NRF) grant funded by the Korea government(MSIT) (No. 2020R1A4A1016649). \\
This work was supported by the UBSI Fellowship Program (project No.1.210131.01) of UNIST. \\
This paper was written while the author was visiting Vitezslav Kala at Charles University, Prague, supported by Czech Science Foundation, grant 21-00420M}
\begin{document}

\maketitle

\begin{abstract}
In this article, we consider the representation of $m$-gonal forms over $\N_0$.
We show that any $m$-gonal forms over $\N_0$ of rank $\ge 5$ is almost regular and ponder the sufficiently large integers which are indeed represented over $\N_0$ among the integers which are locally represented.
And as its consequential result, we prove a finiteness theorem for universal (original polygonal number version's) $m$-gonal forms over $\N_0$.
\end{abstract}

\section{Introduction}
With a long and splendid history {\it polygonal number} is a number which is defined as a total number of dots of a regular polygon.
Especially, the total number of dots of regular $m$-gon with $x$ dots for each side follows the formula
\begin{equation} \label{m number}
    P_m(x)=\frac{m-2}{2}(x^2-x)+x
\end{equation}
and we call the $m$-gonal number as {\it $x$-th $m$-gonal number}.
When we talk about the representation by polygonal numbers, a typically mentioned reference is the famous Fermat's polygonal number conjecture which states that every positive integer may be written as a sum of at most $m$ $m$-gonal numbers.
Which was resolved by Lagrange for $m=4$ in 1770, by Gauss for $m=3$ in 1796, and finally by Cauchy for all $m \ge 3$ in 1813.

As a generalization of the Fermat's Conjecture, the question of classifying the tuples $(a_1,\cdots,a_n) \in \N^n$ for which for any $N \in \N$, there is a solution $(x_1,\cdots,x_n) \in \N_0^n$ satisfying
$$N=a_1P_m(x_1)+\cdots+a_nP_m(x_n)$$
was raised.
As a generalization of Gauss's work, for $m=3$, Liouville classified $(a_1,a_2,a_3) \in \N^3$ such as above.
And as a generalization of Lagrange's work, for $m=4$, Ramanujan classified $(a_1,a_2,a_3,a_4) \in \N^4$ such as above.
But one mistake was found later in the Ramanujan's list.

We call a weighted sum of $m$-gonal numbers, i.e., 
\begin{equation} \label{m form}
    a_1P_m(x_1)+\cdots+a_nP_m(x_n)
\end{equation}
where $(a_1,\cdots,a_n) \in \N^n$ as {\it m-gonal form}.
We simply write the $m$-gonal form of \eqref{m form} as $\left<a_1,\cdots,a_n\right>_m$.
In the case that $m=4$, we conventionally adopt the notation $\left<a_1,\cdots,a_n\right>$ for the diagonal quadratic form (i.e., square form) $\left<a_1,\cdots,a_n\right>_4$. 
Even though, the original definition of $m$-gonal number admits only positive integer $x$ (which is a number of dots for each side of a regular polygon) in \eqref{m number}, recently, as a kind of generalization of the $m$-gonal number, by many authors, the negative integers also often be admitted for the $x$ in \eqref{m number}.
And we call the $P_m(x)$ where $x \in \z$ as {\it generalized $m$-gonal number}.

For a positive integer $N \in \N$, if the diophantine equation 
\begin{equation} \label{rep N}
    F_m(\mathbf x)=a_1P_m(x_1)+\cdots+a_nP_m(x_n)=N
\end{equation}
has a solution $\mathbf x \in \z^n$ (resp. $\N_0^n$), then we say that the $m$-gonal form {\it (globally) represents $N$ over $\z$ (resp. $\N_0$)}.
And when an $m$-gonal form represents every positive integer $N$ over $\z$ (resp. $\N_0$), we say that the $m$-gonal form is {\it universal over $\z$ (resp. $\N_0$)}.
Determining the universality (i.e., the representability of 'every' positive integer) of a given $m$-gonal form is not easy in general.
As a struggling to take a step to closer to the problem, we suggest to verify weaker version's alternative congruence equation
\begin{equation} \label{loc rep N}
    F_m(\mathbf x)=a_1P_m(x_1)+\cdots+a_nP_m(x_n) \eq N \pmod{r}.
\end{equation}
If the congruence equation has a solution $\mathbf x \in \z$ (resp. $\N_0$) for every $r \in \z$, then we say that $F_m(\mathbf x)$ {\it locally represents $N$ over $\z$ (resp. $\N_0$)}.
As you can probably notice, the local representability over $\z$ agrees with the local representability over $\N_0$.
Unlike the global situation, one may completely classify all the $N \in \N$ which is locally represented by arbitrary given form $F_m(\mathbf x)$ without difficulty based on the very useful fact that
a positive integer $N \in \N$ is locally represented by $\left<a_1,\cdots,a_n\right>_m$ if and only if the equation \eqref{rep N} has a $p$-adic integer solution $\mathbf x \in \z_p^n$ for every prime $p$.
For a prime $p$, for a $p$-adic integer $N \in \z_p$,
when the equation $N=a_1P_m(x_1)+\cdots+a_nP_m(x_n)$ has a $p$-adic integer solution $\mathbf x \in \z_p^n$, we say that $\left<a_1,\cdots,a_n\right>_m$ {\it represents $N$ over $\z_p$}.

Obviously, the global representability implies the local representability.
But the converse does not hold in general.
So the struggling does not perfectly work for our purpose, but the local-to-global approach have achieved somewhat success 'over $\z$'.
When the converse also holds, on the other words, when an $m$-gonal form $F_m(\mathbf x)$ (globally) represents every positive integer $N \in \N$ which is locally represented by $F_m(\mathbf x)$ over $\z$ (resp. $\N_0$), we say that 
$F_m(\mathbf x)$ is {\it regular over $\z$ (resp. $\N_0$)}. 
When an $m$-gonal form $F_m(\mathbf x)$ represents every positive integer $N \in \N$ which is locally represented by $F_m(\mathbf x)$ but finitely many over $\z$ (resp. $\N_0$), we say that 
$F_m(\mathbf x)$ is {\it almost regular over $\z$ (resp. $\N_0$)}. 
By Theorem 4.9 (1) in \cite{CO}, the $m$-gonal form (more  precisely, quadratic polynomial whose quadratic part is positive definite) of rank $\ge 5$ represents all sufficiently large integers which are locally represented by the given form over $\z$. 
On the other words, any $m$-gonal form of rank $\ge 5$ represents almost all integers among the integers which are locally represented by the form.
But such a local-to-global principle 'over $\N_0$' still remains hidden to us.
Because the $m$-gonal form over $\N_0$ admits much strict variables than the $m$-gonal form over $\z$, one may naturally expect that it would be having a difficulty in considering the representation over $\N_0$ than over $\z$.
In this article, firstly, we consider a local-to-global principle of $m$-gonal form over $\N_0$.
The following theorem is the our first main goal in this article.

\vskip 0.8em

\begin{thm} \label{main thm}
Any $m$-gonal form $\left<a_1,\cdots,a_n\right>_m$ of rank $n \ge 5$ represents the positive integers $N \in \N_0$ over $\N_0$ provided that
\begin{equation} \label{main eq}
    \begin{cases}
N \ge N(a_1,\cdots,a_n)\cdot(m-2)^3 \\
N \text{ is locally represented by }F_m(\mathbf x)
    \end{cases}
\end{equation}
where $N(a_1,\cdots,a_n)>0$ is a constant which is dependent only on $a_1,\cdots,a_n$.
The cubic on $m$ of \eqref{main eq} is optimal in this sense.
\end{thm}

\vskip 0.8em

We prove the above theorem in section $3$. 
Theorem \ref{main thm} says that any $m$-gonal form of rank $ \ge 5$ is almost regular over $\N_0$.
On the other words, any $m$-gonal form $\left<a_1,\cdots,a_n\right>_m$ of rank $n \ge 5$ (globally) represents all sufficiently integers ($\ge N(a_1,\cdots,a_n)\cdot(m-2)^3$) which are locally represented by the form.
Moreover, Theorem \ref{main thm} gives an answer about the sufficiently large integers too.
For a given $n$-tuple $(a_1,\cdots,a_n) \in \N$ with $n\ge 5$, let $N_{(a_1,\cdots,a_n);m}>0$ be the optimal (i.e., the minimal) integer satisfying that the
$m$-gonal form $\left<a_1,\cdots,a_n\right>_m$ represents the positive integers $N$ over $\N_0$ provided that 
\begin{equation} 
    \begin{cases}
N \ge N_{(a_1,\cdots,a_n);m} \\
N \text{ is locally represented by }F_m(\mathbf x).
    \end{cases}
\end{equation}
From the simple observation that the smallest $m$-gonal number is $m$ (which is increasing as $m$ increases) except $1$, one may induce that $N_{(a_1,\cdots,a_n);m}$ would be asymptotically increasing as $m$ increases.
In virtue of Theorem \ref{main thm}, we may exactly know that the information about the growth of the $N_{(a_1,\cdots,a_n);m}$ which is cubic on $m$.

Based on the Bhargava's escalator tree method, one may easily induce a finiteness theorem for universal $m$-gonal forms over $\z$ for all $m \ge 3$ by using Theorem 4.9 (1) in \cite{CO} which gives that any $m$-gonal form of rank $\ge5$ is almost regular.
A finiteness theorem for universal $m$-gonal forms states that there is the unique and minimal $\gamma_m$ for which the universality (i.e., the representability of every positive integer) of $m$-gonal form over $\z$ is characterized by the representability of only finitely many positive integers $1,2,\cdots,\gamma_m$ by the form.
On the other words, if an $m$-gonal form $\left<a_1,\cdots,a_n\right>_m$ represents every positive integer up to $\gamma_m$ over $\z$, then $\left<a_1,\cdots,a_n\right>_m$ is universal over $\z$.
Since determining the universality of a given form is not easy in general, Conway and Scheeneeberger's announcement of {\it the fifteen theorem} : the first appearance of a finiteness theorem for universal forms was so stunning because a finiteness theorem for universal forms is surprisingly simple criteria to determine the universality. 

An interesting problem about the growth of $\gamma_m$ (which is asymptotically increasing value as $m$ increases) was firstly  questioned by Kane and Liu \cite{KL} and they showed that for any $\epsilon >0$, there is a constant $C_{\epsilon}>0$ for which
$$m-4 \le \gamma_m \le C_{\epsilon}m^{7+\epsilon}$$
holds for any $m \ge 3$.
After which, Kim and the author \cite{KP'} improved their result by showing that the growth of $\gamma_m$ is exactly linear on $m$, i.e., there is a constant $C>0$ for which
$$m-4\le \gamma_m \le C(m-2)$$
for any $m \ge 3$.

As the same stream with the above, in section $4$, we consider the above results over $\N_0$.
Theorem \ref{main thm} may quickly induce a finiteness theorem for universal $m$-gonal forms over $\N_0$ for all $m \ge 3$ and also give an (not optimal) answer about the growth of the size of the set of finitely many positive integers whose representability by an $m$-gonal form over $\N_0$ classify the universality over $\N_0$.
The consequential result is the following theorem.
\vskip 0.8em

\begin{thm} \label{main thm'}
For any $m \ge 3$, there is the unique and minimal $\gamma_{m;\N_0}>0$ for which if an $m$-gonal form represents every positive integer up to $\gamma_{m;\N_0}$ over $\N_0$, then the $m$-goanl form represents every positive integer over $\N_0$.
And the growth of $\gamma_{m;\N_0}$ (which is asymptotically increasing as $m$ increases) is bounded by a cubic on $m$, i.e., there is an absolute constant $C>0$ for which $$m \le \gamma_{m;\N_0} \le C(m-2)^3.$$
\end{thm}

\vskip 0.8em

\begin{rmk}
Theorem \ref{main thm'} does not give the exact growth of $\gamma_{m;\N_0}$ on $m$.
The problem claming the exact growth of $\gamma_{m;\N_0}$ (or giving a better bound) on $m$ could be an interesting problem.
\end{rmk}

\vskip 0.8em

In this article, we adopt the arithmetic theory of quadratic forms.
Any unexplained notation and terminology can be found in \cite{O}.

\vskip 0.8em


\section{Preliminaries}

The $m$-gonal form $\left<a_1,\cdots,a_n\right>_m$ represents an integer $A(m-2)+B$ over $\z, \z_p$, or $\N_0$, namely,
there is $\mathbf x \in \z^n, \z_p^n$, or $\N_0^n$ for which
$$A(m-2)+B=\frac{m-2}{2}((a_1x_1^2+\cdots+a_nx_n^2)-(a_1x_1+\cdots+a_nx_n))+(a_1x_1+\cdots+a_nx_n)$$
if and only if there is $k \in \z, \z_p$, or $\N_0$ for which the system
\begin{equation} \label{main system}
\begin{cases}
    a_1x_1^2+\cdots+a_nx_n^2=2A+B+k(m-4)\\
    a_1x_1+\cdots+a_nx_n=B+k(m-2)\\
\end{cases}
\end{equation}
is solvable over $\z, \z_p$, or $\N_0$, respectively.
In order to consider the representations of $m$-gonal form over $\z$ or $\z_p$, the author already have been concentrated a discussion about the $k \in \z$ or $\z_p$ of \eqref{main system} in \cite{rank 5}.
As an extension of the work in \cite{rank 5}, in this article, we again consider the system \eqref{main system} and the $k$ in \eqref{main system} which admits not only integer solution $\mathbf x \in \z$ but also non-negative integer solution $\mathbf x \in \N_0^n$ to understand the representation of $m$-gonal forms over $\N_0$.

Very roughly speaking, when $2A+B+k(m-4)$ is large, if $B+k(m-2)$ is not large enough, then there would not exist a non-negative integer solution $\mathbf x$ for \eqref{main system}.
The other way, if $B+k(m-2)$ is sufficiently large, then there would be a lot of chance that the system \eqref{main system} has a non-negative integer solution (but because of the Cauchy-Schwarz inequality, the $B+k(m-2)$ could not be freely large).

The following thankful proposition makes possible to expand the discussion which was used to consider the representation of $m$-gonal form over $\z$ to our work which is a representation of $m$-gonal form over $\N_0$ in this article.

\vskip 0.8em

\begin{prop} \label{prop 1}
Suppose that $\max\limits_{1\le i \le n}\sqrt{a_i}\sqrt{\alpha} \le \beta$. 
Then for any $(x_1,\cdots,x_n) \in \mathbb R^n$ satisfying
\begin{equation}
    \begin{cases}
    a_1x_1^2+\cdots+a_nx_n^2=\alpha\\
    a_1x_1+\cdots+a_nx_n=\beta,
    \end{cases}
    \end{equation}
    we have that $x_i\ge 0$ for all $1 \le i \le n$.
\end{prop}
\begin{proof}
Note that the hyperbolic plane $a_1x_1+\cdots+a_nx_n=\beta$ intersects with the sphere 
$a_1x_1^2+\cdots+a_nx_n^2=\alpha$ only on the space $(\mathbb R^{+}\cup \{0\})^n$.
\end{proof}

\vskip 0.8em

We recall a simple observation that the system \eqref{main system} holds for $ \mathbf x \in \z, \N_0$, or $\z_p$ if and only if
the equation
\begin{equation} \label{eq2}
    \left(B+k(m-2)-\left(\sum \limits_{i=2}^na_ix_i\right)\right)^2+\sum\limits_{i=2}^na_1a_ix_i^2=2Aa_1+Ba_1+k(m-4)a_1 
\end{equation}
holds with $x_1=\frac{1}{a_1}\left(B+k(m-2)-\left(\sum \limits_{i=2}^na_ix_i\right)\right)$.
The equation \eqref{eq2} may be organized as
\begin{equation} \label{Qaa}
Q_{a_1 ; \mathbf a}(\mathbf x-(B+k(m-2))\mathbf r)=(2A+B+k(m-4))a_1-(B+k(m-2))^2 \cdot \left(1-\sum \limits _{i=2}^na_ir_i\right) 
\end{equation}
where $Q_{a_1 ; \mathbf a}(x_2,\cdots,x_n):=\sum_{i=2}^n(a_1a_i+a_i^2)x_i^2+\sum_{2\le i<j \le n}2a_ia_jx_ix_j$ is a positive definite quadratic form and $r_2,\cdots,r_n \in \q$ are the solution for
$$\begin{cases}
(a_1a_2+a_2^2)r_2+a_2a_3r_3+\cdots+a_2a_nr_n=a_2\\
a_2a_3r_2+(a_1a_3+a_3^2)r_3 +\cdots + a_3a_nr_n=a_3 \\
\quad \quad \quad \quad \quad \quad \quad \quad \quad \vdots \\
a_2a_nr_2+a_3a_nr_3+\cdots +(a_1a_n+a_n^2)r_n=a_n
\end{cases}$$
(in practice, $r_2=\cdots=r_n=\frac{1}{a_1+\cdots+a_n}$).

The above simple observation may admit to consider the diophantine quadratic equation \eqref{Qaa} of rank $n-1$ instead of the diophantine system \eqref{main system} of rank $n$.

So from now on, we mainly consider the diophantine quadratic equation \eqref{Qaa} of rank $n-1$ instead of the diophantine system \eqref{main system} of rank $n$.


\section{Representation of $m$-gonal form over $\N_0$}

Lemma \ref{3 lem} is the most critical argument in this article.
When an integer $A(m-2)+B$ is locally represented by an $m$-gonal form $\left<a_1,\cdots,a_n\right>_m$, there are  $k \in \z$ for which the system \eqref{main system} is locally solvable (i.e., the system \eqref{main system} has an $p$-adic integer solution $\mathbf x_p \in \z_p^n$ for every prime $p$).
Moreover, the $k$'s have very regular distribution (precisely, the set of every $k \in \z$ for which the diophantine system \eqref{main system} is locally solvable is a finite union of arithmetic sequences).
The regular distribution's merit is that it makes easily pick fitting $k \in \z$ in \eqref{main system} which meets our purpose.

\vskip 0.8em

\begin{lem} \label{3 lem}
For an $m$-gonal form $\left<a_1,\cdots,a_n\right>_m$ of rank $n\ge5$, let an integer $A(m-2)+B$ with $0\le B \le m-3$ be locally represented by $\left<a_1,\cdots,a_n\right>_m$.
Then for some residue
$$k(A,B) \in \z/K(\mathbf a)\z$$
and
$$P=\prod \limits_{p \in T(\mathbf a)\cup\{2\}}p^{s(p)}
$$
with $0\le s(p) \le \frac{1}{2}\ord_p(4a_1)$,
the quadratic equation 
\begin{equation} \label{lem eq}
    Q_{a_1;\mathbf a}(P\mathbf x-(B+k'(m-2))\mathbf r)=(2A+B+k'(m-4))a_1-(B+k'(m-2))^2\cdot\left(1-\sum_{i=2}^na_ir_i\right)
\end{equation}
is locally primitively solvable
for any $k' \eq k \pmod{K(\mathbf a)}$
where $K(\mathbf a)=K(a_1,\cdots,a_n)$ is a constant which is dependent only on $a_1,\cdots,a_n$ and $T(\mathbf a)=T(a_1,\cdots,a_n)$ is a finite set of all odd primes $p$ for which there are at most four units of $\z_p$ in $\{a_1,\cdots,a_n\}$ by admitting a recursion.
\end{lem}
\begin{proof}
One may use Propositions 3.2, 3.4, 3.6, 3.8, and 3.9 in \cite{rank 5} to prove this lemma.
\end{proof}

\vskip 0.8em

Now, we are already ready to prove our first main goal.

\vskip 0.8em

\begin{proof}[proof of Theorem \ref{main thm}]
Without loss of generality, we assume that $a_1 \le \cdots \le a_n$.
By Theorem 4.9 (2) in \cite{CO}, we may obtain a constant $C_{a_1;\mathbf a}>0$ for which
$$Q_{a_1;\mathbf a}(P\mathbf x-(B+k(m-2))\mathbf r)=C_{a_1;\mathbf a}$$
has an integer solution $\mathbf x \in \z^{n-1}$ with $n-1 \ge 4$ where $P|\prod \limits_{p \in T(\mathbf a)\cup\{2\}}p^{\frac{1}{2}\ord_p(4a_1)}$
provided that
$$\begin{cases}
Q_{a_1;\mathbf a}(P\mathbf x-(B+k(m-2))\mathbf r)=N \text{ is primitively locally solvable} \\
N>C_{a_1;\mathbf a} \text{ is sufficiently large}.
\end{cases}$$
Note that such a constant $C_{a_1;\mathbf a}$ is dependent only on $a_1,\cdots,a_n$.

Now for an integer $A(m-2)+B$ which is locally represented by $\left<a_1,\cdots,a_n\right>_m$, our attention go to $k \in \N_0$ satisfying
\begin{equation} \label{k}
\begin{cases}
k\eq k(A,B) \pmod{K(\mathbf a)} \\
(2A+B+k(m-4))a_1-(B+k(m-2))^2\cdot\left(1-\sum_{i=2}^na_ir_i\right)>C_{a_1;\mathbf a} \\
\sqrt{a_n}\sqrt{2A+B+k(m-4)} \le B+k(m-2)
\end{cases}    
\end{equation}
where $k(A,B)$ is a residue in $\z/K(\mathbf a) \z$ in Lemma \ref{3 lem}.
The second inequality of \eqref{k} draw
$$\alpha_{A,B;m}^-<k<\alpha_{A,B;m}^+$$
where $r_i=\frac{1}{a_1+\cdots+a_n}$ and  $2(m-2)^2\alpha_{A,B;m}^{\pm}:= \left(\sum \limits_{i=1}^na_i\right)(m-4)-2B(m-2) \pm$ $ \sqrt{\left\{\left(\sum \limits_{i=1}^na_i\right)(m-4)-2B(m-2)\right\}^2+4(m-2)^2\left\{\left(\sum \limits_{i=1}^na_i\right)(2A+B)-B^2-\frac{C_{a_1;\mathbf a}\sum \limits_{i=1}^na_i}{a_1}\right\}}$.
And the third inequality of \eqref{lem eq} draw
$$k<\beta_{A,B;m}^- \ \text{ or } \ \beta_{A,B;m}^+<k$$
where 
$\beta_{A,B;m}^{\pm}:=\frac{a_n(m-4)-2B(m-2)\pm\sqrt{\{a_n(m-4)-2B(m-2)\}^2+4(m-2)^2\{a_n(2A+B)-B^2\}}}{2(m-2)^2}$.
Therefore if $$\beta_{A,B;m}^++K(\mathbf a)<\alpha_{A,B;m}^+,$$ then we may get a $k \in \N_0$ satisfying the three conditions in \eqref{k}.
Through elementary but dirty calculation, one may obtain
a constant $N(a_1,\cdots,a_n)>0$ which is dependent only on $a_1,\cdots,a_n$
satisfying that
$$\beta_{A,B;m}^++K(\mathbf a)<\alpha_{A,B;m}^+$$
with $0 \le B \le m-3$ holds for any $A \ge N(a_1,\cdots,a_n) \cdot (m-2)^2$ and $m \ge 3$.

Consequently, we may conclude that for an integer $A(m-2)+B$ which is locally represented by $\left<a_1,\cdots,a_n\right>_m$ with 
$$\begin{cases}
A \ge N(a_1,\cdots,a_n) \cdot(m-2)^2\\
0 \le B \le m-3,
\end{cases}$$
we may take $k \in \N_0$ satisfying \eqref{k}.
Then from the first and second conditions of \eqref{k}, we obtain that $A(m-2)+B$ is represented by $\left<a_1,\cdots,a_n\right>_m$ over $\z$, i.e., there is $\mathbf x \in \z^n$ for which $$A(m-2)+B=a_1P_m(x_1)+\cdots+a_nP_m(x_n)$$ holds, more precisely, the diophantine system
$$
\begin{cases}
    a_1x_1^2+\cdots+a_nx_n^2=2A+B+k(m-4)\\
    a_1x_1+\cdots+a_nx_n=B+k(m-2)\\
\end{cases}
$$
holds.
And from the third condition of \eqref{k}, we obtain that the $\mathbf x $ is indeed in $\N_0^n$ by Proposition \ref{prop 1}.
This completes the first argument.

The Cauchy-Schwarz inequality gives that 
$$(a_1+\cdots+a_n)(a_1x_1^2+\cdots+a_nx_n^2) \le (a_1x_1+\cdots+a_nx_n)^2.$$
Which induces that for a given $(a_1,\cdots,a_n) \in \N^n$, the cubic on $m$ is optimal in this argument.
\end{proof}

\vskip 0.8em

\begin{rmk}
Note that the local representability of $m$-gonal form over $\N_0$ is coincide with the local representability of $m$-gonal form over $\z$.
The global representability of $m$-gonal form over $\N_0$ obviously implies the global representability of $m$-gonal form over $\z$, but converse does not hold.
In \cite{non alm reg}, the author considered infinitely many $m$-gonal forms of rank $4$ which is not almost regular over $\z$.
So we may conclude that there are also infinitely many $m$-gonal form of rank $4$ which is not almost regular over $\N_0$.
So the rank $\ge 5$ in Theorem \ref{main thm} is optimal.
\end{rmk}

\section{Finiteness theorem for universal $m$-gonal forms
(original polygonal number version)}

We start this section by recalling the Bhargava's escalaltor tree.
We call the smallest positive integer which is not represented by a non-universal $m$-gonal form $\left<a_1,\cdots,a_n\right>_m$ over $\N_0$ as the {\it the truant of $\left<a_1,\cdots,a_n\right>_m$}.
The {\it escalator tree of $m$-gonal forms over $\N_0$} is one sided tree with root $\emptyset_m$.
Conventionally, the truant of the $m$-gonal form $\emptyset_m$ is considered as the smallest positive integer $1$.
If an $m$-gonal form $\left<a_1,\cdots,a_n\right>_m$ of a node is not universal (i.e., has the truant), then we connect the node with the nodes $\left<a_1,\cdots,a_{n+1}\right>_m$ which are escalated superforms of $\left<a_1,\cdots,a_n \right>_m$ to represent the truant of $\left<a_1,\cdots,a_n \right>_m$.
Note that by the construction, the rank of every $m$-gonal form of depth $n$ of the escalator tree is $n$.
In order to avoid the appearance of the same forms, we set assumption $a_1\le \cdots \le a_n \le a_{n+1}$.
If an $m$-gonal form on a node is universal (i.e., does not have truant), then the node become a leaf of the tree.

So, universal forms of the tree appear only on the leaves and the universal $m$-gonal forms on the leaves would be kind of proper universal forms.
And one may notice that a univesal $m$-gonal form would contains at least one (proper universal) $m$-gonal form on a leaf of the escalator tree of $m$-gonal form as its subform.

Once one completes the escalator tree of $m$-gonal forms over $\N_0$, then one should get the $\gamma_{m;\N_0}$ as the largest truant of the tree.

Since the smallest positive $m$-gonal number (over $\N_0$) is $m$ except $1$, if $a_1+\cdots+a_n<m-1$, then the $m$-gonal form $\left<a_1,\cdots,a_n\right>_m$ would not represent $a_1+\cdots+a_n+1$ (more precisely, all the integers between $a_1+\cdots+a_n+1$ and $m-1$).
So for an $m$-gonal form $\left<a_1,\cdots,a_n\right>_m$ on a node of the escalator tree, in the case that $a_1+\cdots+a_n<m-1$, its truant would be $a_1+\cdots+a_n+1$ and so its children would be forms of
$$\left<a_1,\cdots,a_n,a_{n+1}\right>_m$$
where $a_n \le a_{n+1} \le a_1+\cdots+a_n+1$.
Which gives a special observation that for a fixed depth $d$, the depth $d$ of the escalator tree of $m$-gonal form take on the same forms for all sufficiently large $m$.
Especially, as an example, when $m \ge 8$, the escalator tree of $m$-gonal forms over $\N_0$ up to $3$ depth would appear as following :

$$
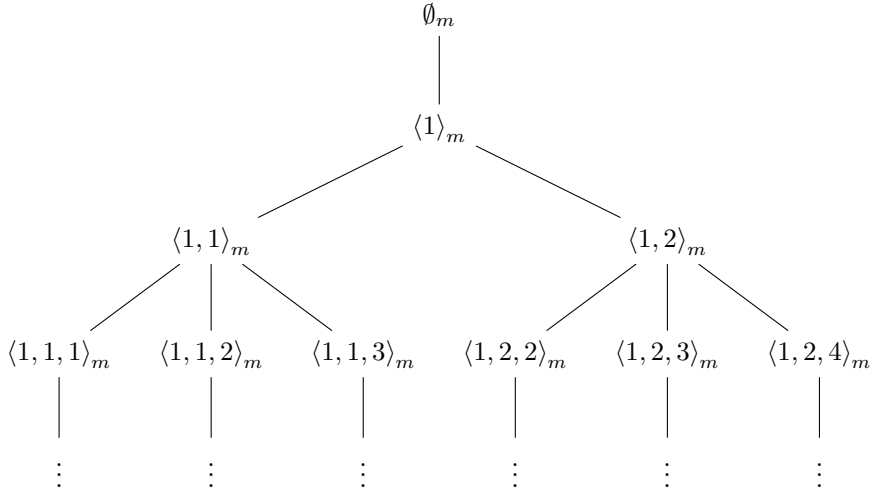
\begin{tikzpicture}
\tikzstyle{level 1}=[sibling distance=80mm] 
\tikzstyle{level 2}=[sibling distance=60mm] 
\tikzstyle{level 3}=[sibling distance=20mm] 
  \node {$\emptyset_m$}
    child {node {$\left<1\right>_m$}
     child {node {$\left<1,1\right>_m$}
       child {node {$\left<1,1,1\right>_m$}
             child {node {$\vdots$}}       }       
       child {node {$\left<1,1,2\right>_m$}
             child {node {$\vdots$}}       }       
       child {node {$\left<1,1,3\right>_m$}
             child {node {$\vdots$}}       }       
    }
     child {node {$\left<1,2\right>_m$}
       child {node {$\left<1,2,2\right>_m$}
             child {node {$\vdots$}}       }       
       child {node {$\left<1,2,3\right>_m$}
             child {node {$\vdots$}}       }       
       child {node {$\left<1,2,4\right>_m$}
             child {node {$\vdots$}}       }       
    }
    };

\end{tikzpicture}$$

\vskip 0.8em

Now we prove a finiteness theorem for univesal $m$-gonal forms over $\N_0$ based on the Bhargava's escalator tree idea.
Just before, we see a method to determine the local representability of $m$-gonal by adopting already well known local representability results of quadratic forms.

\vskip 0.8em

\begin{prop} \label{loc.rep}
Let $F_m(\mathbf x)=a_1P_m(x_1)+\cdots+a_nP_m(x_n)$ be a primitive $m$-gonal form.
\begin{itemize}
    \item [(1) ] When $p$ is an odd prime with $p|m-2$, $F_m(\mathbf x)$ is universal over $\z_p$.
    \item [(2) ] When $m \not\eq 0 \pmod 4$, $F_m(\mathbf x)$ is universal over $\z_2$.
   \item [(3) ] When $p$ is an odd prime with $(p,m-2)=1$, an integer $N$ is represented by $F_m(\mathbf x)$ over $\z_p$ if and only if the integer $8(m-2)N+(a_1+\cdots+a_n)(m-4)^2$ is represented by the diagonal quadratic form $\left<a_1,\cdots,a_n \right>$ over $\z_p$.
    \item [(4) ] When $m \eq 0 \pmod 4$, an integer $N$ is represented by $F_m(\mathbf x)$ over $\z_2$ if and only if the integer $\frac{m-2}{2}N+(a_1+\cdots+a_n)\left(\frac{m-4}{4}\right)^2$ is represented by the diagonal quadratic form $\left<a_1,\cdots,a_n \right>$ over $\z_2$.
\end{itemize}
\end{prop}
\begin{proof}
See Proposition 3.1 in \cite{rank 5}.
\end{proof}

\vskip 0.8em

\begin{rmk} \label{loc rep rmk}
By Proposition \ref{loc.rep}, when a diagonal quadratic form $\left<a_1,\cdots,a_n\right>$ is locally universal, the $m$-gonal forms
$$\left<a_1,\cdots,a_n\right>_m$$ whose coefficients coincide with the coefficients of the quadratic form $\left<a_1,\cdots,a_n\right>$ are also locally universal for all $m \ge 3$.
\end{rmk}

\vskip 0.8em

\begin{proof}[proof of Theorem \ref{main thm'}]
Note that for $m \ge 31$, the escalator tree of $m$-gonal forms over $\N_0$ up to $5$ depth would be independent on $m$ and
all the candidates for the coefficients of $m$-gonal forms of depth $5$ would be the tuples in the finite set
$$T_{d=5}:=\{(a_1,\cdots,a_5) \in \N^5|a_1=1, a_i \le a_{i+1} \le a_1+\cdots+a_i+1\}.$$
One may check that the quadratic form $\left<a_1,\cdots,a_5\right>$ is locally universal for each $(a_1,\cdots,a_5) \in T_{d=5}$ case by case.
And then by Remark \ref{loc rep rmk}, we may obtain that all the $m$-gonal forms of depth $5$ on the nodes of depth $5$ of escalator tree of $m$-gonal forms are locally universal for all $m \ge 31$.

From Theorem \ref{main thm}, which implies that for all $m \ge 31$, the $m$-gonal forms $\left<a_1,\cdots,a_5\right>_m$ on the nodes of depth $5$ of escalator tree of $m$-gonal forms represents every sufficiently large integer 
$$N\ge C_{\ge31}(m-2)^3$$
where $C_{\ge31}:=\max\{N(a_1,\cdots,a_5)|(a_1,\cdots,a_5)\in T_{d=5}\}$.
Which says that  the truant of the escalator trees of $m$-gonal forms over $\N_0$ could not be bigger than $C_{\ge 31}(m-2)^3$ for all $m \ge 31$!,
giving that $$\gamma_{m;\N_0} \le C_{\ge 31}(m-2)^3$$ for all $m \ge 31$.

For $3 \le m \le 30$ too, the local-to-global theorem (Theorem \ref{main thm}) may direcly imply the existence of $\gamma_{m;\N_0}$.
Consequently, we conclude that for 
$$C:=\max \left\{\frac{\gamma_{m;\N_0}}{(m-2)^3}|3 \le m \le 30 \right\}\cup \{C_{\ge31}\},$$
the theorem holds.
\end{proof}

\end{document}